\documentclass[11pt
]{amsart}
\usepackage{amssymb,amscd, epsfig, mathrsfs, xypic, amsmath,amsthm, color}
\usepackage{hyperref}
\usepackage{blindtext}
\usepackage{enumerate}
\usepackage{eepic}

\xyoption{all}
\newtheorem{thm}{Theorem}[section] 

\newtheorem{lem}[thm]{Lemma}
\newtheorem{prop}[thm]{Proposition} 
\newtheorem{df-pr}[thm]{Definition-Proposition}

\theoremstyle{definition}
\newtheorem{defn}[thm]{Definition}

\newtheorem{exm}[thm]{Example}

\newcommand{\ZZ}{{\mathbb Z}}

\newcommand{\bfs}{{\mathbf s }}

\newcommand{\calF}{{\mathcal F}}
\newcommand{\calG}{{\mathcal G}}

\newcommand{\scA}{{\mathscr A}}

\newcommand{\bbN}{{\mathbb N}}

\newcommand{\id}{{\operatorname{id}}}

\newcommand{\FSVT}{{\it FSVT}}

\newcommand{\Ess}{{\mathcal{E}ss}}

\newsavebox{\savepar}

\numberwithin{equation}{section}

\newcounter{labelflag} 
\newcommand{\labelon}{\setcounter{labelflag}{1}}
\newcommand{\Label}[1]{\ifnum\thelabelflag=1\ifmmode
\makebox[0in][l]{\qquad\fbox{\rm#1}} \else
\marginpar{\vspace{0.7\baselineskip} \hspace{-1.1\textwidth}
\fbox{\rm#1}} \fi \fi \label{#1} } \labelon

\begin{document} 
\title{Flagged Grothendieck polynomials}
\author{Tomoo Matsumura}
\maketitle 

\begin{abstract}
We show that the \emph{flagged Grothendieck polynomials} defined as generating functions of flagged set-valued tableaux of Knutson--Miller--Yong \cite{KnutsonMillerYong} can be expressed by a Jacobi--Trudi type determinant formula generalizing the work of Hudson--Matsumura \cite{HudsonMatsumura2}. We also introduce the \emph{flagged skew Grothendieck polynomials} in these two expressions and show that they coincide.
\end{abstract}

\section{Introduction}
Lascoux--Sch{\"u}tzenberger (\cite{Lascoux1}, \cite{LascouxSchutzenberger3}) introduced the Grothendieck polynomials to represent the $K$-theory classes of the structure sheaves of Schubert varieties and Fomin--Kirillov  (\cite{GrothendieckFomin}, \cite{DoubleGrothendieckFomin}) gave their combinatorial description in terms of pipe dreams or rc graphs. Knutson--Miller--Yong \cite{KnutsonMillerYong} expressed the Grothendick polynomial associated to a \emph{vexillary permutation} as the generating function of \emph{flagged set-valued tableaux}, unifying the work of Wachs  \cite{Wachs} on flagged tableaux, and Buch \cite{BuchLRrule} on set-valued tableaux ({\it cf.} \cite{KnutsonMillerYong2}). On the other hand, in the joint work \cite{HudsonMatsumura2} with Hudson, the author proved the Jacobi--Trudi type formula for the vexillary Grothendieck polynomials in the context of degeneracy loci formula, generalizing his joint work \cite{HIMN} and \cite{HudsonMatsumura} with Hudson, Ikeda, and Naruse for the Grassmannian case (see also \cite{Anderson2016} and \cite{Matsumura2016}). 

Motivated by these results, we study  the generating functions of flagged set-valued tableaux in general beyond the ones given by vexillary permutations. For a given partition $\lambda=(\lambda_1\geq \cdots \geq\lambda_r>0)$ of length $r$, a \emph{flagging} $f$ of $\lambda$ is a nondecreasing sequence of natural numbers $(f_1,\dots,f_r)$. A \emph{flagged set-valued tableau of shape $\lambda$ with a flagging $f$} is nothing but a set-valued tableaux of shape $\lambda$ of Buch \cite{BuchLRrule} satisfying extra conditions that the numbers used in the $i$-th row are at most $f_i$ for all $i$. Let $\FSVT(\lambda,f)$ be the set of all flagged set-valued tableaux  of shape $\lambda$ with the flagging $f$. Let $x=(x_1,x_2,\dots )$ be a set of infinitely many indeterminants. Following Knutson--Miller--Yong's work, we define the \emph{flagged Grothendieck polynomials} $G_{\lambda,f}(x)$ by
\[
G_{\lambda,f}(x) := \sum_{T \in \FSVT(\lambda, f)} \beta^{|T| - |\lambda|} \prod_{k\in T} x_{k}.
\]

The main goal of this paper is to show that $G_{\lambda,f}(x)$ is given by the following Jacobi-Trudi type determinant formula (Theorem \ref{thmMain}) :
\begin{equation}\label{introeq1}
G_{\lambda,f}(x) = \det\left(  \sum_{s=0}^{\infty} \binom{i-j}{s}\beta^s G_{\lambda_i+j-i+s}^{[f_i]}(x) \right)_{1\leq i,j \leq r}.
\end{equation}
Here $G_m^{[p]}(x)$ is defined by the generating function 
\[
\sum_{m\in \ZZ} G_m^{[p]}(x) u^m = \frac{1}{1+\beta u^{-1}} \prod_{1\leq i \leq {p}}\frac{1+\beta x_i}{1-x_iu},
\]
which geometrically corresponds to the Segre classes of vector bundles (see \cite{HIMN}). Since not all flagged Grothendieck polynomials are Grothendieck polynomials of Lascoux and Sch{\"u}tzenberger, our result generalizes the work of Knutson--Miller--Yong and Hudson--Matsumura mentioned above. 

The \emph{flagged Schur polynomials} are the specialization of the flagged Grothendieck polynomials $G_{\lambda,f}(x)$ at $\beta=0$. These polynomials were introduced by Lascoux and Sch{\"u}tzenberger in \cite{SchubertLascoux} to identify Schubert polynomials associated to vexillary permutations. They are also generalizations of Schur polynomials and satisfy the generalized Jacobi--Trudi formula, due to Gessel and Wachs. In this paper, we closely follow Wachs' inductive proof in \cite{Wachs}, which makes the use of the divided difference operators. In particular, our proof shows that the above determinant formula in terms of the one row Grothendieck polynomials $G_m^{[p]}(x)$ behaves nicely under the action of those symmetrizing operators. 


In Section \ref{Groth}, we give an inductive proof that the Grothendieck polynomial associated to a vexillary permutation is a flagged Grothendieck polynomial. This gives a combinatorial  proof of Knutson--Miller--Yong's result and of the determinant formula in \cite{HudsonMatsumura2}. We also show that any flagged Grothendieck polynomial can be obtained from a monomial by applying the divided difference operators. 

In Section \ref{SecSkew}, we introduce the flagged \emph{skew} Grothendieck polynomials. Their tableaux and determinant expressions are natural extension of the (row) flagged skew Schur functions studied by Wachs \cite{Wachs}. We show those two expressions coincide. 

It is worth pointing out that the formulas of Knutson--Miller--Yong and Hudson--Matsumura are also for \emph{double} Grothendieck polynomials defined for two sets of variables. Thus it would be natural to extend our result to its double version (see the work \cite{ChenLiLouck} of Chen--Li--Louck for the flagged double Schur functions). This extension will be studied elsewhere.
\section{Flagged Grothendieck polynomials}\label{mainsec}
A {\it partition} is a weakly decreasing finite sequence $\lambda=(\lambda_1,\dots,\lambda_r)$ of positive integers and $r$ is called its length. We often identify a partition $\lambda$ of length $r$ with its Young diagram $\{(i,j) \ |\ 1\leq i\leq r, 1\leq j\leq \lambda_i\}$ in English notation. A \emph{flagging} $f$ of a partition $\lambda$ of length $r$ is a weakly increasing sequence $f=(f_1,\dots,f_r)$ of positive integers. A \emph{flagged set-valued tableau} $T$ of shape $\lambda$ with a flagging $f$ is a set-valued tableau of shape $\lambda$ such that each filling in the $i$-th row consists of the numbers not greater than $f_i$. Namely, $T$ is a filling of the boxes of $\lambda$ such that the box at $(i,j)$ in $\lambda$ is filled by a non-empty subset of $\{1,\dots,f_i\}$ and the maximum number at $(i,j)$ is at most the minimum number at $(i,j+1)$ for $j+1\leq \lambda_i$ and less than the minimum number at $(i+1,j)$ for $j\leq \lambda_{i+1}$. In particular, if $f_1=\cdots = f_r$, then it coincides with the definition of set-valued tableaux defined by Buch in \cite{BuchLRrule}. We denote by $\FSVT(\lambda,f)$ the set of all flagged set-valued tableaux of shape $\lambda$ with the flagging $f$. If $f_1=f_2=1$ and $r>1$, then $\FSVT(\lambda,f)=\varnothing$. Thus we assume that if $f_1=1$ and $r>1$, then $f_2>1$. 
\begin{exm}
Let $\lambda=(2,1)$ and $f=(2,4)$. Then $\FSVT(\lambda,f)$ contains tableaux such as

\begin{center}
\setlength{\unitlength}{0.5mm}
\begin{picture}(30,30)
\linethickness{0.7pt}
\put(0,20){\line(1,0){20}}
\put(0,10){\line(1,0){20}}
\put(0,0){\line(1,0){10}}
\put(0,20){\line(0,-1){20}}
\put(10,20){\line(0,-1){20}}
\put(20,20){\line(0,-1){10}}
\put(3,13){$1$}
\put(13,13){$1$}
\put(1,3){$23$}
\end{picture}
\begin{picture}(30,30)
\linethickness{0.7pt}
\put(0,20){\line(1,0){20}}
\put(0,10){\line(1,0){20}}
\put(0,0){\line(1,0){10}}
\put(0,20){\line(0,-1){20}}
\put(10,20){\line(0,-1){20}}
\put(20,20){\line(0,-1){10}}
\put(1,13){$12$}
\put(13,13){$2$}
\put(1,3){$34$}
\end{picture}
\begin{picture}(30,30)
\linethickness{0.7pt}
\put(0,20){\line(1,0){20}}
\put(0,10){\line(1,0){20}}
\put(0,0){\line(1,0){10}}
\put(0,20){\line(0,-1){20}}
\put(10,20){\line(0,-1){20}}
\put(20,20){\line(0,-1){10}}
\put(3,13){$1$}
\put(11,13){$12$}
\put(1,3){$23$}
\end{picture}
\begin{picture}(30,30)
\linethickness{0.7pt}
\put(0,20){\line(1,0){20}}
\put(0,10){\line(1,0){20}}
\put(0,0){\line(1,0){10}}
\put(0,20){\line(0,-1){20}}
\put(10,20){\line(0,-1){20}}
\put(20,20){\line(0,-1){10}}
\put(3,13){$2$}
\put(13,13){$2$}
\put(3,3){$4$}
\end{picture}.
\end{center}
If we chance $f$ to $f'=(2,3)$, then $\FSVT(\lambda,f')$ doesn't contain the second and forth tableaux.
\end{exm}

Let $x=(x_1,x_2, \dots)$ be the set of infinitely many variables. Let $\ZZ[\beta]$ be the polynomial ring of the variable $\beta$ where we set $\deg \beta =-1$. Let $\ZZ[\beta][x]$ and $\ZZ[\beta][[x]]$ be the ring of polynomials and of formal power series in $x$ respectively. We define the \emph{flagged Grothendieck polynomial} associated to a partition $\lambda$ and a flagging $f$ by
\begin{equation}\label{dfGtab}
G_{\lambda,f}(x) := \sum_{T \in \FSVT(\lambda, f)} \beta^{|T| - |\lambda|} \prod_{k\in T} x_k,
\end{equation}
where $|T|$ is the total number of entries in $T$, $|\lambda|$ is the number of boxes in $\lambda$, and $k\in T$ denotes an entry in $T$.
We also define an element $\widetilde{G}_{\lambda,f}(x)$ of $\ZZ[\beta][[x]]$ by
\begin{equation}\label{df}
\widetilde{G}_{\lambda,f}(x) := \det\left(  \sum_{s=0}^{\infty} \binom{i-j}{s}\beta^s G_{\lambda_i+j-i+s}^{[f_i]}(x) \right)_{1\leq i,j \leq r},
\end{equation}
where $G_m^{[p]}(x) \in \ZZ[\beta][[x]]$ is defined by the generating function
\begin{equation}\label{dfGm}
G^{[p]}(x;u)=\sum_{m\in \ZZ} G_m^{[p]}(x) u^m = \frac{1}{1+\beta u^{-1}} \prod_{1\leq i \leq {p}}\frac{1+\beta x_i}{1-x_iu}.
\end{equation}
For the empty partition, we set both of the functions $G_{\lambda,f}(x)$ and $\widetilde{G}_{\lambda,f}(x)$ to be $1$. We also remark that $G_{-m}^{[p]}=(-\beta)^{m}$ and $G_m^{[1]} = x_1^{m}$ for all integer $m\geq 0$, which follow from the direct computation.

At $\beta = 0$, $G_{\lambda,f}(x)$ specializes to the flagged Schur polynomial of Wachs in \cite{Wachs}, and $\widetilde{G}_{\lambda,f}(x)$ is nothing but the corresponding Jacobi--Trudi formula since $G_m^{[p]}(x)$ becomes the complete symmetric function of degree $m$.

The main goal of this section is to show that $G_{\lambda,f}(x)$ and $\widetilde{G}_{\lambda,f}(x)$ coincide (Theorem \ref{thmMain}). We closely follow Wachs' proof of the analogous statement for the flagged Schur polynomials in \cite{Wachs}. The key for the proof is the action of the divided difference operators, which makes the induction proof possible.

\subsection{Divided difference operators and basic formulas}
We recall the definition of divided difference operators and show a few formulas that will be used in the proof of the propositions to follow.

Let $S_n$ be the permutation group of the set $\{1,\dots,n\}$. We have the action of $S_n$ on $\ZZ[\beta][[x]]$ permuting the variables. For example, let $s_i$ denote the $i$-th transposition {\it i.e.} $s_i(i)=i+1, s_i(i+1)=i$ and $s_i(j) = j$ for $j\not=i,i+1$, then $s_i(f(x))$ is defined by exchanging $x_i$ and $x_{i+1}$ in $f(x) \in \ZZ[\beta][[x]]$ and leave other variables unchanged.
\begin{defn}
For an element $f(x)$ of $\ZZ[\beta][[x]]$, we define
\[
\pi_i(f(x)) := \frac{(1+\beta x_{i+1})f(x) - (1+\beta x_i)s_i(f(x))}{x_i-x_{i+1}}.
\]
\end{defn}
The following Leibniz rule can be checked by a direct computation: for $f(x), g(x) \in \ZZ[\beta][[x]]$, we have
\begin{equation}\label{Leibniz}
\pi_i(fg) = \pi_i(f) g+ s_i(f) \pi_i(g) + \beta s_i(f) g.
\end{equation}
It is also easy to check directly that, if $f(x)$ is symmetric in $x_i$ and $x_{i+1}$, then we have
\begin{equation}\label{lem2-3}
\pi_i(x_i^k f(x)) = 
\begin{cases}
-\beta f(x)& (k=0) \\
\left(\displaystyle\sum_{s=0}^{k-1} x_i^s x_{i+1}^{k-1-s}  + \beta \displaystyle\sum_{s=1}^{k-1} x_i^{s} x_{i+1}^{k-s}\right)f(x) & (k>0).
\end{cases}
\end{equation}
\begin{lem}\label{lem2-4}
For each $m\in \ZZ$ and $p\in \ZZ_{\geq 1}$, we have
\[
\pi_i (G_m^{[p]}(x)) = 
\begin{cases}
G_{m-1}^{[p+1]}(x)& (i=p),\\
-\beta G_m^{[p]}(x)& (i\not=p).
\end{cases}
\]
\end{lem}
\begin{proof}
If $i=p$, it follows from the identity $\pi_pG^{[p]}(x;u) = u G^{[p+1]}(x;u)$ which can be proved by a direct computation. If $i\not=p$, then $G_m^{[p]}(x)$ is symmetric in $x_i$ and $x_{i+1}$, and hence (\ref{lem2-3}) implies the claim. 
\end{proof}
\begin{lem}\label{lem2-5}
If $f(x)$ is symmetric in $x_{p}$ and $x_{p+1}$, then we have $\pi_p(G_{m}^{[p]} f) = G_{m-1}^{[p+1]}  f$.
\end{lem}
\begin{proof}
It follows from the Leibniz rule (\ref{Leibniz}) and Lemma \ref{lem2-4}.
\end{proof}
\begin{lem}\label{lem2-6}
For each $m\in \ZZ$ and $p\in \ZZ_{\geq 1}$, we have
\[
G_{m}^{[p]} 
- \frac{x_1}{1+\beta x_1}G_{m-1}^{[p]}
- \frac{x_1}{1+\beta x_1}\beta G_{m}^{[p]} 
=\left.G_{m}^{[p]}\right|_{x_1=0}.
\]
\end{lem}
\begin{proof}
It follows from the identity
\begin{eqnarray*}
\sum_{m\in \ZZ} \left(G_{m}^{[p]} 
- \frac{x_1}{1+\beta x_1}G_{m-1}^{[p]}
- \frac{x_1}{1+\beta x_1}\beta G_{m}^{[p]}\right)u^m
= \left.G^{[p]}(x;u)\right|_{x_1=0},
\end{eqnarray*}
which can be checked by a direct computation.
\end{proof}
\subsection{The main theorem}
We prove the main theorem (Theorem \ref{thmMain}) below by induction based on the following two propositions.
\begin{prop}\label{propJ1}
Let $\lambda$ be a partition of length $r$ with a flagging $f$. If $\lambda_1>\lambda_2$ and $f_1<f_2$, then we have
\begin{enumerate}
\item[$(\mathrm{i})$] $\pi_{f_1}(\widetilde{G}_{\lambda,f})= \widetilde{G}_{\lambda',f'}$,
\item[$(\mathrm{ii})$] $\pi_{f_1}(G_{\lambda,f})= G_{\lambda',f'}$,
\end{enumerate}
where $\lambda'=(\lambda_1-1,\lambda_2,\dots,\lambda_r)$ and $f'=(f_1+1,f_2,\dots,f_r)$.
\end{prop}
\begin{proof}
For (i), we recall from \cite[\S 3.6]{HIMN} that we can write
\begin{equation}\label{HIMNeq}
\widetilde{G}_{\lambda,f} = \sum_{\bfs \in \ZZ^r} a_{\bfs}   G_{\lambda_1+s_1}^{[f_1]}\cdots G_{\lambda_r+s_r}^{[f_r]},
\end{equation}
where $a_{\bfs} \in \ZZ[\beta]$ is the coefficient of $t^{\bfs}$ in the Laurent series expansion
\[
\prod_{1\leq i<j\leq r} (1-\bar t_i/\bar t_j)   = \sum_{\bfs=(s_1,\dots, s_r) \in \ZZ^r} a_{\bfs} t_1^{s_1}\cdots t_r^{s_r}.
\]
Here we denoted $\bar t = \frac{-t}{1+\beta t}= -t \sum_{s\geq 0} (-\beta)^st^s$.
Since $f_1<f_2$, one can apply Lemma \ref{lem2-5} to the expression (\ref{HIMNeq}) and obtains (i). Indeed, we have
\begin{eqnarray*}
\pi_{f_1}(\widetilde{G}_{\lambda,f})
= \sum_{\bfs \in \ZZ^r} a_{\bfs}   \pi_{f_1}\left(G_{\lambda_1+s_1}^{[f_1]}\cdots G_{\lambda_r+s_r}^{[f_r]}\right)
=\sum_{\bfs \in \ZZ^r} a_{\bfs}   G_{\lambda_1-1+s_1}^{[f_1+1]}G_{\lambda_2+s_2}^{[f_2]}\cdots G_{\lambda_r+s_r}^{[f_r]}=\widetilde{G}_{\lambda',f'}.
\end{eqnarray*}

Next we prove (ii). Let $t:=f_1$ and $t':=f_1+1$. Define an equivalence relation $\sim$ on $\FSVT(\lambda,f)$ as follows: for $T_1,T_2 \in \FSVT(\lambda,f)$, let $T_1 \sim T_2$ if the collection of boxes that contain either $t$ or $t'$ is the same for $T_1$ and $T_2$. We have 
\[
G_{\lambda,f}=\sum_{\scA \in \FSVT(\lambda,f)/\!\sim}\left(\sum_{T \in \scA} M(T)\right),
\]
where $M(T):= \beta^{|T| - |\lambda|} \prod_{k\in T} x_k$ for each $T \in \FSVT(\lambda,f)$. Let $\scA$ be the equivalence class whose tableaux have the configuration of $t$ and $t'$ as shown in Figure 1.
\setlength{\unitlength}{0.5mm}
\begin{center}
\begin{picture}(180,100)
\thicklines
\put(0,90){\line(1,0){195}}
\put(80,80){\line(1,0){115}}
\put(80,70){\line(1,0){90}}
\put(80,60){\line(1,0){30}}
\put(160,50){\line(1,0){30}}
\put(140,30){\line(1,0){20}}
\put(0,0){\line(1,0){140}}
\put(173,83){\small{$t t\dots t$}}
\put(145,83){\small{$t \dots t$}}
\put(145,73){\small{$t'\dots t'$}}
\put(85,73){\small{$t \dots t$}}
\put(85,63){\small{$t'\dots t'$}}
\put(115,73){\small{$*\dots *$}}
\put(178,75){\tiny{$m_1$}}
\put(150,65){\tiny{$r_1$}}
\put(120,65){\tiny{$m_2$}}
\put(90,55){\tiny{$r_2$}}
\put(73,48){{$\cdot$}}
\put(68,43){{$\cdot$}}
\put(63,38){{$\cdot$}}
\put(39,23){\small{$*\dots *$}}
\put(15,23){\small{$t\dots t$}}
\put(13.5,13){\small{$t'\dots t'$}}
\put(45,16){\tiny{$m_k$}}
\put(19.5,6){\tiny{$r_k$}}
\put(10,30){\line(1,0){50}}
\put(10,20){\line(1,0){50}}
\put(10,10){\line(1,0){25}}
\put(10,30){\line(0,-1){20}}
\put(35,30){\line(0,-1){20}}
\put(60,30){\line(0,-1){10}}
\put(0,90){\line(0,-1){90}}
\put(195,90){\line(0,-1){10}}
\put(170,90){\line(0,-1){20}}
\put(140,90){\line(0,-1){20}}
\put(110,80){\line(0,-1){20}}
\put(80,80){\line(0,-1){20}}
\put(190,80){\line(0,-1){30}}
\put(160,50){\line(0,-1){20}}
\put(140,30){\line(0,-1){30}}
\end{picture}
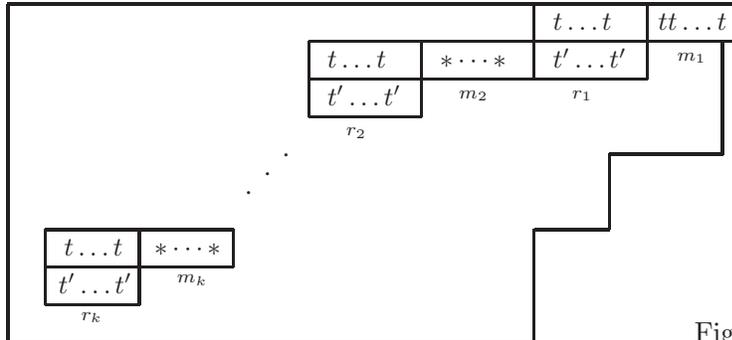
Figure 1.
\end{center}
Each rectangle with $*$ has $m_i$ boxes and each box contains $t$ or $t'$ so that the total number of entries $t$ and $t'$ in the rectangle is $m_i$ or $m_i+1$ for $i\geq 2$. Note that $m_i$ and $r_i$ may be $0$ and the rectangles in Figure 1 may not be connected. Then we see that 
\[
\sum_{T \in \scA} M(T) = x_t^{m_1}(x_tx_{t'})^{r_1+\cdots+ r_k} \left( \prod_{i=2}^k \left( \sum_{v=0}^{m_i} x_t^v x_{t'}^{m_i-v} + \beta\sum_{v=1}^{m_i} x_t^v x_{t'}^{m_i+1-v}\right)   \right)R(\scA),
\]
where $R(\scA)$ is the polynomial contributed from  the entries other than $t$ and ${t'}$. Let
\[
R'(\scA):=\left( \prod_{i=2}^k \left( \sum_{v=0}^{m_i} x_t^v x_{t'}^{m_i-v} + \beta\sum_{v=1}^{m_i} x_t^v x_{t'}^{m_i+1-v}\right)   \right)R(\scA).
\]
Observe that the factor $(x_tx_{t'})^{r_1+\cdots+ r_k} R'(\scA)$ is symmetric in $x_t$ and $x_{t'}$. Thus, by Lemma \ref{lem2-3}, if $m_1=0$, then we have $r_1=0$ and 
\begin{equation}\label{eqtab1}
\pi_t\left(\sum_{T \in \scA} M(T)\right) = -\beta (x_tx_{t'})^{r_2+\cdots+ r_k} R'(\scA),
\end{equation}
if $m_1=1$, we have
\begin{equation}\label{eqtab2}
\pi_t\left(\sum_{T \in \scA} M(T)\right) 
= (x_tx_{t'})^{r_1+\cdots+ r_k} R'(\scA),
\end{equation}
and if $m_1\geq 2$, we have
\begin{eqnarray}\label{eqtab3}
\pi_t\left(\sum_{T \in \scA} M(T)\right) 
&=& \left(\sum_{s=0}^{m_1-1} x_t^s x_{t'}^{m_1-1-s}  + \beta \sum_{s=1}^{m_1-1} x_t^{s} x_{t'}^{m_1-s}\right)
(x_tx_{t+1})^{r_1+\cdots+ r_k} R'(\scA).
\end{eqnarray}
We consider the decomposition
\[
\FSVT(\lambda,f)/\!\!\sim \ \ = \calF_1 \sqcup \calF_2\sqcup \calF_3\sqcup \calF_4
\]
where $\calF_1,\dots, \calF_4$ are the sets of equivalence classes whose configurations of the boxes containing $t$ or $t'$ respectively satisfy
\begin{itemize}
\item[(1)] $m_1=0$ (so that $r_1=0$),
\item[(2)] $m_1=1$ and the box at $(1,\lambda_1)$ in $\lambda$ contains more than one entry (so that $r_1=0$),
\item[(3)] $m_1=1$ and the box at $(1,\lambda_1)$ in $\lambda$ contains contains only $t$,
\item[(4)] $m_1\geq 2$.
\end{itemize}
By the expressions (\ref{eqtab1}) and (\ref{eqtab2}), we have
\[
\sum_{\scA \in \calF_1\sqcup \calF_2}\pi_t\left(\sum_{T \in \scA} M(T)\right) = 0.
\]
Now we consider the equivalence class $\scA'$ in $\FSVT(\lambda',f')/\!\sim$ whose associated skew diagram of boxes containing $t$ or $t'$ is as shown in Figure 2 below. 
\setlength{\unitlength}{0.5mm}
\begin{center}
\begin{picture}(180,100)
\thicklines
\put(0,90){\line(1,0){190}}
\put(80,80){\line(1,0){110}}
\put(80,70){\line(1,0){90}}
\put(80,60){\line(1,0){30}}
\put(160,50){\line(1,0){30}}
\put(140,30){\line(1,0){20}}
\put(0,0){\line(1,0){140}}
\put(173,83){\small{$t\dots t$}}
\put(145,83){\small{$t \dots t$}}
\put(145,73){\small{$t'\dots t'$}}
\put(85,73){\small{$t \dots t$}}
\put(85,63){\small{$t'\dots t'$}}
\put(115,73){\small{$*\dots *$}}
\put(172,75){\tiny{$m_1\!-\!1$}}
\put(150,65){\tiny{$r_1$}}
\put(120,65){\tiny{$m_2$}}
\put(90,55){\tiny{$r_2$}}
\put(73,48){{$\cdot$}}
\put(68,43){{$\cdot$}}
\put(63,38){{$\cdot$}}
\put(39,23){\small{$*\dots *$}}
\put(15,23){\small{$t\dots t$}}
\put(13.5,13){\small{$t'\dots t'$}}
\put(45,16){\tiny{$m_k$}}
\put(19.5,6){\tiny{$r_k$}}
\put(10,30){\line(1,0){50}}
\put(10,20){\line(1,0){50}}
\put(10,10){\line(1,0){25}}
\put(10,30){\line(0,-1){20}}
\put(35,30){\line(0,-1){20}}
\put(60,30){\line(0,-1){10}}
\put(0,90){\line(0,-1){90}}
\put(190,90){\line(0,-1){10}}
\put(170,90){\line(0,-1){20}}
\put(140,90){\line(0,-1){20}}
\put(110,80){\line(0,-1){20}}
\put(80,80){\line(0,-1){20}}
\put(190,80){\line(0,-1){30}}
\put(160,50){\line(0,-1){20}}
\put(140,30){\line(0,-1){30}}
\end{picture}
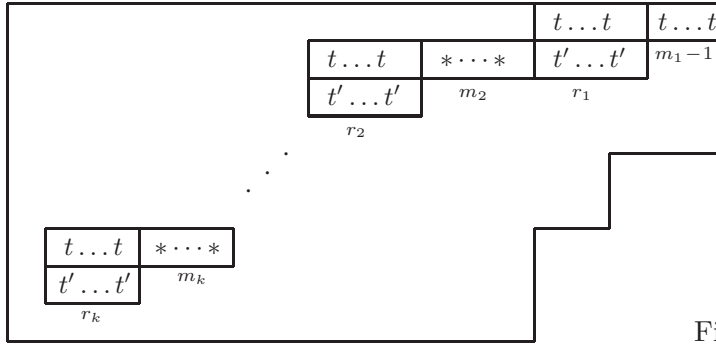
Figure 2.
\end{center}
One can see that, if $m_1=1$, then $\sum_{T\in \scA'} M(T)$ is exactly the right hand side of (\ref{eqtab2}) under the condition (3), and if $m_1\geq 2$, then  $\sum_{T\in \scA'} M(T)$ is exactly the right hand side of (\ref{eqtab3}). It is also clear that $\calF_3$ and $\calF_4$ are in bijection to the equivalence classes of $\FSVT(\lambda',f')$ such that $m_1=1$ and $m_1\geq 2$ respectively. Thus the desired identity holds.
\end{proof}
\begin{prop}\label{propJ2}
Let $\lambda$ be a partition of length $r$ with a flagging $f$. If $f_1=1$, then we have
\begin{itemize}
\item[$(\mathrm{i})$] $\widetilde{G}_{\lambda,f} = x_1^{\lambda_1}(\widetilde{G}_{\lambda',f'}|_{x_1=0})$,
\item[$(\mathrm{ii})$] $G_{\lambda,f} = x_1^{\lambda_1}(G_{\lambda',f'}|_{x_1=0})$,
\end{itemize}
where $\lambda'=(\lambda_2,\dots,\lambda_r)$ and $f'=(f_2,\dots,f_r)$.
\end{prop}
\begin{proof}
First we observe that (ii) holds clearly since $f_1=1$. 
We prove (i). 
%
Since $f_1=1$, we see that the first row of the determinant for $\widetilde{G}_{\lambda,f}$ is
\[
\left(x_1^{\lambda_1}, x_1^{\lambda_1} \frac{x_1}{1+\beta x_1},\dots, x_1^{\lambda_1} \left(\frac{x_1}{1+\beta x_1}\right)^{r-1}\right).
\]
Indeed, since $G_m^{[1]} = x_1^{m}$ for $m \geq 0$, we have
\[
\sum_{s=0}^{\infty} \binom{1-j}{s}\beta^s G_{\lambda_1+j-1+s}^{[f_1]}(x) = x_1^{\lambda_1+j-1}\sum_{s=0}^{\infty} \binom{1-j}{s}\beta^s x_1^{s} = \frac{x_1^{\lambda_1+j-1}}{(1+\beta x_1)^{j-1}}.
\]
We do the column operation to $\widetilde{G}_{\lambda,f}$ by subtracting $\frac{x_1}{1+\beta x_1}$ times the $(j-1)$-st column from the $j$-th column for each $j=2,\dots,r$ so that the first row becomes $(x_1^{\lambda_1},0,\dots, 0)$. Then we observe that the $(i,j)$-entry for $i,j\geq 2$ equals to
\begin{eqnarray*}
&&\sum_{s=0}^{\infty} \binom{i-j}{s}\beta^s G_{\lambda_i+j-i+s}^{[f_i]} - \frac{x_1}{1+\beta x_1}\sum_{s=0}^{\infty} \binom{i-j+1}{s}\beta^sG_{\lambda_i+j-1-i+s}^{[f_i]}\\
&=&\sum_{s=0}^{\infty} \binom{i-j}{s}\beta^s G_{\lambda_i+j-i+s}^{[f_i]} - \frac{x_1}{1+\beta x_1}\sum_{s=0}^{\infty} \left(\binom{i-j}{s}+\binom{i-j}{s-1}\right)\beta^sG_{\lambda_i+j-1-i+s}^{[f_i]}\\
&=&\sum_{s=0}^{\infty} \binom{i-j}{s}\beta^s \left(G_{\lambda_i+j-i+s}^{[f_i]} 
- \frac{x_1}{1+\beta x_1}G_{\lambda_i+j-1-i+s}^{[f_i]}
- \frac{x_1}{1+\beta x_1}\beta G_{\lambda_i+j-i+s}^{[f_i]}\right)\\
&=&\left.\left(\sum_{s=0}^{\infty} \binom{i-j}{s}\beta^s G_{\lambda_i+j-i+s}^{[f_i]}\right)\right|_{x_1=0},
\end{eqnarray*}
where the first equality is by an identity of the binomial coefficients and the last equality follows from Lemma \ref{lem2-6}. Finally the cofactor expansion with respect to the first row gives the desired identity for (i).
\end{proof}
\begin{thm}\label{thmMain}
For each partition $\lambda$ and a flagging $f$, we have $G_{\lambda,f}(x)  = \widetilde{G}_{\lambda,f}(x)$. In particular, $\widetilde{G}_{\lambda,f}(x)$ is a polynomial in $x$.
\end{thm}
\begin{proof}
By induction on $(r, |f|:=f_1+\cdots + f_r)$, Proposition \ref{propJ1} and \ref{propJ2} imply the claim: for the base case $(r,|f|) = (0,0)$, the claim holds trivially. If $f_1=1$, apply Proposition \ref{propJ2} and if $f_1>1$, apply Proposition \ref{propJ1}. In both cases, the claim follows from the induction hypothesis.
\end{proof}
\section{Grothendieck polynomials}\label{Groth}
In this section, we show that the Grothendieck polynomials associated to a vexillary permutation is a flagged Grothendieck polynomial, recovering the results of Knutson--Miller--Yong \cite{KnutsonMillerYong} and Hudson--Matsumura \cite{HudsonMatsumura2}. We also show that any flagged Grothendieck polynomial can be obtained from a monomial in the same way as any Grothendieck polynomial is defined.

The Grothendieck polynomial $\calG_w=\calG_w(x_1,\dots,x_n)$ associated to a permutation $w\in S_n$ is defined as follows. We use the same convention as in \cite{BuchLRrule}. For the longest element $w_0$ in $S_n$, we set
\[
\calG_{w_0} = x_1^{n-1}x_2^{n-2}\cdots x_{n-2}^2x_{n-1}  = \prod_{i=1}^{n-1} x_i^{n-i}.
\]
If $w$ is not the longest, there is $i$ such that $\ell(ws_i)=\ell(w) +1$. We then define
\[
\calG_w:= \pi_i(\calG_{ws_i}).
\]
This definition is independent of the choice of $s_i$ because the operators $\pi_i$ satisfy the Coxeter relations. By the same reason we can define $\pi_w$ by $\pi_w = \pi_{i_k}\cdots \pi_{i_1}$ where $w=s_{i_1}\cdots s_{i_k}$ with $\ell(w)=k$.  

Now we recall how to obtain a partition $\lambda(w)$ and a flagging $f(w)$ for each vexillary permutation $w\in S_n$. We follow \cite{FlagsFulton} and \cite{KnutsonMillerYong} ({\it cf.} \cite{HudsonMatsumura2}). Let $r_w$ be the \emph{rank function} of $w \in S_n$ defined by $r_w(p,q) := \sharp\{ i\leq p  |\ w(i) \leq q\}$ and we define the diagram $D(w)$ of $w$ by
\[
D(w):=\{(p,q)\in \{1,\dots,n\}\times \{1,\dots,n\} \ |\ \pi(p) >q, \ \mbox{and} \ \pi^{-1}(q)>p\}.
\]
We call an element of the grid $\{1,\dots,n\}\times \{1,\dots,n\}$ a box. The \emph{essential set} $\Ess(w)$ of $w$ is the subset of $D(w)$ given by
\[
\Ess(w):=\{(p,q) \ |\ (p+1,q), (p,q+1) \not\in D(w)\}.
\]
A permutation $w \in S_n$ is called \emph{vexillary} if it avoids the pattern $(2143)$, \textit{i.e.} there is no $a<b<c<d$ such that $w(b)<w(a)<w(d)<w(c)$. In \cite{Wachs}, a vexillary permutation was called a \emph{single-shape} permutations. Fulton showed in \cite{FlagsFulton} that $w \in S_n$ is vexillary if and only if the boxes in $\Ess(w)$ are placed along the direction going from northeast to southwest. We can assign a partition $\lambda(w)$ to each vexillary permutation $w$ as follows: let the number of boxes $(i,i+k)$ in the $k$-th diagonal of the Young diagram of $\lambda(w)$ be equal to the number of boxes in the $k$-th diagonal of $D(w)$ for each $k$ (see \cite{KnutsonMillerYong, KnutsonMillerYong2}). This defines a bijection $\phi$ from $D(w)$ to $\lambda$, namely $\phi(p,q) = (p-r_w(p,q), q-r_w(p,q))$ for each $(p,q)\in D(w)$. In particular, $\phi$ restricted to $\Ess(w)$ is a bijection onto the set of the southeast corners of $\lambda(w)$.
Let $r$ be the length of $\lambda(w)$. The flagging $f(w)=(f(w)_1,\dots,f(w)_r)$ associated to $w$ is defined as follows. 
We can choose a subset $\{(p_i,q_i),i=1,\dots,r\}$ of $\{1,\dots,n\}\times \{1,\dots,n\}$ containing $\Ess(w)$ and satisfying 
\begin{eqnarray}
&&p_1\leq p_2 \leq \cdots \leq p_r,  \ \ \ q_1\geq q_2 \geq \cdots \geq q_r, \label{flagging1}\\
&&p_i - r_w(p_i,q_i) = i, \ \ \ \forall i=1,\dots,r. \label{flagging2} 
\end{eqnarray}
In \cite{HudsonMatsumura2}, we called this subset $\{(p_i,q_i)\}$ a \emph{flagging set} of $w$ and used it to express the \emph{double} Grothendieck polynomials as a determinant. We set $f(w)$ by letting $f(w)_i:=p_i$. We can always express $\lambda(w)$ by $\lambda_i = q_i-p_i +i$ for each $i=1,\dots,r$. Remark that the set $\FSVT(\lambda(w),f(w))$ doesn't depend of the choice of flagging sets. 
\begin{exm}
Consider a vexillary permutation $w=(w(1)\cdots w(5))=(23541)$ in $S_5$. We represent the corresponding permutation matrix $M_w$ by $(M_w)_{ij} = \delta_{w(i),j}$. In the picture below, we represent $1$ in $M_w$ by a dot and the boxes in $D(w)$ by squares. We make hooks by drawing lines from each dot going south and east, and then $D(w)$ is the collection of boxes that are not on the hooks. We see that $\lambda(w)=(\lambda_1,\dots,\lambda_4) = (2,1,1,1)$. For a flagging set of $w$, we must have $(p_1,q_1) = (3,4)$ and $(p_4,q_4) = (4,1)$ since they consist $\Ess(w)$. Then the conditions (\ref{flagging1}) and (\ref{flagging2}) require that we must choose $(p_2,q_2)$ from $(3,2)$ and $(4,3)$ and $(p_3,q_3)$ from $(3,1)$ and $(4,2)$ in such a way that $p_2\leq p_3$ and $q_2\geq q_3$. Thus $f(w)=(3,3,3,4)$, $(3,3,4,4)$ or $(3,4,4,4)$. By the column strictness, each of the flaggings gives the same collection of flagging set-valued tableaux.

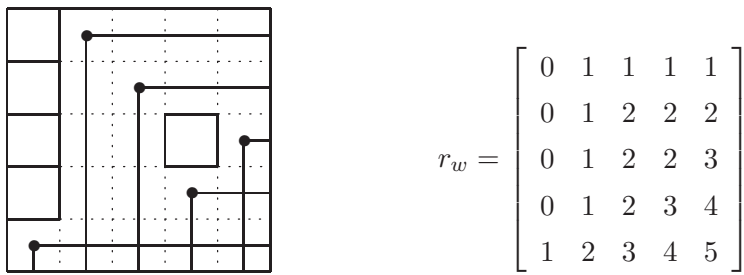
\setlength{\unitlength}{0.7mm}
\begin{center}
\begin{picture}(80,35)
\linethickness{0.7pt}
\put(0,30){\line(1,0){50}}
\put(50,30){\line(0,-1){50}}
\put(0,30){\line(0,-1){50}}
\put(0,-20){\line(1,0){50}}
\dottedline{2}(0.1,20)(50,20)
\dottedline{2}(0.1,10)(50,10)
\dottedline{2}(0.1,0)(50,0)
\dottedline{2}(0.1,-10)(50,-10)
\dottedline{2}(10,30)(10,-20)
\dottedline{2}(20,30)(20,-20)
\dottedline{2}(30,30)(30,-20)
\dottedline{2}(40,30)(40,-20)

\put(3.8,-16.4){$\bullet$}
\put(13.8,23.6){$\bullet$}
\put(23.8,13.6){$\bullet$}
\put(33.8,-6.4){$\bullet$}
\put(43.8,3.6){$\bullet$}

\put(5,-15){\line(0,-1){5}}
\put(15,25){\line(0,-1){45}}
\put(25,15){\line(0,-1){35}}
\put(35,-5){\line(0,-1){15}}
\put(45,5){\line(0,-1){25}}

\put(5.1,-15){\line(1,0){45}}
\put(15.1,25){\line(1,0){35}}
\put(25.1,15){\line(1,0){25}}
\put(35.1,-5){\line(1,0){15}}
\put(45.1,5){\line(1,0){5}}

\put(10,30){\line(0,-1){40}}
\put(0,-10){\line(1,0){10}}
\put(0,0){\line(1,0){10}}
\put(0,10){\line(1,0){10}}
\put(0,20){\line(1,0){10}}

\put(30,10){\line(0,-1){10}}
\put(40,10){\line(0,-1){10}}
\put(30,10){\line(1,0){10}}
\put(30,0){\line(1,0){10}}
\end{picture}
$r_w=
\left[
\begin{array}{ccccc}
0&1&1&1&1\\
0&1&2&2&2\\
0&1&2&2&3\\
0&1&2&3&4\\
1&2&3&4&5
\end{array}
\right]$
\end{center}
\end{exm}

\vspace{2mm}

\begin{thm}
If $w$ is a vexillary permutation, then $\calG_w=G_{\lambda(w), f(w)}$.
\end{thm}
\begin{proof}
By Propositions \ref{propJ1} and \ref{propJ2}, the proof is almost identical to the one given by Wachs for the flagged Schur functions in \cite{Wachs}. We write the proof for completeness since we have slightly different terminologies and notations.

First we observe that for the longest element $w_0 \in S_n$, we have $\lambda(w_0)=(n-1,n-2,\dots,1)$ and $f(w_0)=(1,2,3,\dots,n-1)$. Therefore by definition we have
\[
G_{\lambda(w_0),f(w_0)} = \prod_{i=1}^{n-1} x_i^{n-i} = \calG_{w_0}.
\]
We show the claim by induction on $(n,\ell(w_0) - \ell(w))$ with the lexicographic order where $\ell(w)$ denotes the length $w\in S_n$. We say that $w$ has a descent at $i$ if $w_i>w_{i+1}$. Let $d$ be the leftmost descent of $w$. We consider the following three cases. 

\emph{Case 1}. Assume $d>1$. Let $w'=w s_{d-1}$, then it is easy to see that $w'$ is vexillary. Since $\ell(w')>\ell(w)$, we have $\calG_{w'}=G_{\lambda(w'), f(w')}$ by the induction hypothesis. We can also observe that $\lambda(w')_1=\lambda(w)_1+1$ and $\lambda(w')_i=\lambda(w)_i$ for $i\geq 2$. Furthermore $f_1(w') = d-1=f_1(w) - 1$ and $f(w')_i=f(w)_i$ for $i\geq 2$. Thus, by Proposition \ref{propJ1}, we have
\[
\calG_w = \pi_{d-1}\calG_{w'} = \pi_{d-1}G_{\lambda(w'), f(w')} = G_{\lambda(w),f(w)}. 
\]
This completes the case $d>1$.

\emph{Case 2}. Assume $d=1$ and $w(1)<n$. Let $w':=s_{w(1)}w$ so that $w'(1)=w(1)+1$ and $w'(l)=w(1)$ for $l$ such that $w(l)=w(1)+1$. It is clear that $w'$ is vexillary.
Since the leftmost descent of $w$ is $1$, we have $\ell(w')=\ell(w)+1$. Consider the vexillary permutation
\[
u=(w(1)+1,w(1),w(1)+2,w(1)+3,\dots,n, w(1)-1,w(1)-2, \dots,1).
\]
Since $w$ is vexillary, all numbers greater than $w(1)$ appear in ascending order in $w$, and therefore we can find integers $i_1,\dots,i_k$ greater than $1$, satisfying $w' s_{i_1}s_{i_2}\cdots s_{i_k} = u$ and $\ell(w')+k=\ell(u)$. For such a choice, we also have $ws_{i_1}s_{i_2}\cdots s_{i_k} s_1 = u$. Thus by definition we have
\[
\calG_w = \pi_{i_1}\cdots \pi_{i_k}\pi_1 \calG_u\ \ \ \mbox{and}\ \ \ 
\calG_{w'} = \pi_{i_1}\cdots \pi_{i_k}\calG_u.
\]
By induction we have $\calG_u = G_{\lambda(u),f(u)}$ and $\calG_{w'}=G_{\lambda(w'),f(w')}$. We observe that 
\[
\lambda(u)=(w(1), w(1)-1,w(1)-1,\dots,w(1)-1, w(1)-2,w(1)-3,\dots,2 ,1)
\]
and $f(u)=(1,2,3,\dots,n)$, from which we find
\[
\calG_u =x_1^{w(1)}x_2^{w(1)-1}\cdots x_{n-w(1)+1}^{w(1)-1}x_{n-w(1)+2}^{w(1)-2} \cdots x_{n-2}^2x_{n-1}^1.
\]
Lemma \ref{lem2-3} implies that $\pi_1\calG_u = (1/x_1) \calG_u$. Since $i_1,\dots,i_k$ are greater than $1$, again Lemma \ref{lem2-3} implies that
\[
\calG_w 
= \pi_{i_1}\cdots \pi_{i_k}((1/x_1) \calG_u) = (1/x_1) \calG_w'=\calG_w = (1/x_1) G_{\lambda(w'),f(w')}.
\]
Furthermore, we observe that
\begin{eqnarray*}
\lambda(w')_1=\lambda(w)_1+1,	&& \lambda(w')_i=\lambda(w)_i,  \ \ \ \forall i=2,\dots,r,\\
f(w')_1=f(w)_1=1,							&& f(w')_i=f(w)_i, \ \ \ \forall  i=2,\dots,r.
\end{eqnarray*}
This implies that $G_{\lambda(w),f(w)}=(1/x_1)G_{\lambda(w'),f(w)}$. Therefore we conclude that $\calG_w=G_{\lambda(w),f(w)}$.

\emph{Case 3}. Assume $d=1$ and $w(1)=n$. Let $w_0w=s_{i_1}\cdots s_{i_k}$ where $\ell(w)+k=\ell(w_0)$. Since $w(1)=n$, it follows that $i_1,\dots,i_k$ are greater than $1$. Consequently, 
\[
\calG_w = \pi_{i_k}\pi_{i_{k-1}}\cdots \pi_{i_1}(x_1^{n-1}x_2^{n-2}\cdots x_{n-2}^2x_{n-1}^1) =x_1^{n-1} \pi_{i_k}\pi_{i_{k-1}}\cdots \pi_{i_1}(x_2^{n-2}\cdots x_{n-2}^2x_{n-1}^1). 
\]
Consider the vexillary permutation $w':=(w(2),\dots,w(n))$ in $S_{n-1}$. By induction, we have $\calG_{w'}=G_{\lambda(w'),f(w')}$. Moreover, by definition we have
\[
\pi_{i_k}\pi_{i_{k-1}}\cdots \pi_{i_1}(x_2^{n-2}\cdots x_{n-2}^2x_{n-1}^1) = (\calG_{w'})^{\dagger}.
\]
where $f(x)^{\dagger}$ denotes the function obtained from $f(x)$ by replacing $x_i$ with $x_{i+1}$. Thus $\calG_w = x_1^{n-1}  (\calG_{w'})^{\dagger}$. Since $(f(w')_2,\dots,f(w')_n) = (f(w)_2-1,\dots, f(w)_n-1)$, it follows that
\[
(G_{\lambda(w'),f(w')})^{\dagger} = \left.G_{\lambda(w'),\widetilde{f(w)}}\right|_{x_1=0},
\]
where $\widetilde{f(w)} = (f_2(w),\dots, f_n(w))$. Thus Proposition \ref{propJ2} implies that 
\[
\calG_w = x_1^{n-1}(G_{\lambda(w'),f(w')})^{\dagger} = x_1^{n-1}\cdot \left(\left.G_{\lambda(w'),\widetilde{f(w)}}\right|_{x_1=0}\right) = G_{\lambda(w),f(w)}.
\]
This completes the proof.
\end{proof}
\begin{thm}
Let $\lambda$ be a partition of length $r$ with a flagging $f$. Then $G_{\lambda,f}$ is equal to $\pi_w(x_1^{a_1}\cdots x_r^{a_r})$ where
\[
a_i=\lambda_i+f_i - i, \ \ \ \mbox{and} \ \ \ w=s_rs_{r+1}\cdots s_{f_r-1}\cdot s_{r-1}s_r\cdots s_{f_{r-1}-1}\cdots s_1s_2\cdots s_{f_1-1}.
\]
\end{thm}
\begin{proof}
We prove by induction on the sum of the flagging $|f|=f_1+\cdots + f_r$. If $|f|=1$, then $f_1=1$ and $r=1$. In this case, we have $G_{\lambda,f} = x_1^{f_1}=\pi_{\id}x_1^{a_1}$. 
Assume that $|f|>1$. If $f_1>1$, consider $\lambda'=(\lambda_1+1,\lambda_2,\dots, \lambda_r)$ and $f'=(f_1-1,f_2,\dots,f_r)$. The induction hypothesis implies that $G_{\lambda',f'} = \pi_{w s_{f_1-1}} (x_1^{a_1}\cdots x_r^{a_r})$. By Proposition \ref{propJ1} and the Coxeter relation of divided differences, we obtain
\[
G_{\lambda,f} =\pi_{f_1-1}G_{\lambda',f'} = \pi_{f_1-1}\pi_{w s_{f_1-1}} (x_1^{a_1}\cdots x_r^{a_r}) = \pi_w(x_1^{a_1}\cdots x_r^{a_r}).
\]
If $f_1=1$, consider $\lambda'=(\lambda_2,\dots,\lambda_r)$ and $f'=(f_2-1,\dots,f_r-1)$. Then we have $G_{\lambda,f} = x_1^{\lambda_1}(G_{\lambda',f'})^{\dagger} =  x_1^{a_1}(G_{\lambda',f'})^{\dagger}$. On the other hand, by induction, we have $G_{\lambda',f'} = \pi_{w'}(x_1^{a_1'}\cdots x_{r-1}^{a_{r-1}'})$ where $a_i' = a_{i+1}$ and $w'=s_{r-1}s_{r}\cdots s_{f_r-2}s_{r-2}s_{r-1}\cdots s_{f_{r-1}-2}\cdots s_1s_2\cdots s_{f_2-2}$. 
We can see that $(\pi_{w'}(x_1^{a_1'}\cdots x_{r-1}^{a_{r-1}'}))^{\dagger} = \pi_w(x_2^{a_2}\cdots x_{r}^{a_{r}})$ and therefore we have
\[
G_{\lambda,f} =x_1^{a_1}\pi_w(x_2^{a_2}\cdots x_{r}^{a_{r}})  = \pi_w(x_1^{a_1}x_2^{a_2}\cdots x_{r}^{a_{r}}),
\]
since none of the transpositions in the expression of $w$ is equal to $s_1$.
\end{proof}
\section{Flagged skew Grothendieck polynomials}\label{SecSkew}
Consider two partitions $\lambda=(\lambda_1\geq \cdots \geq \lambda_r>0)$ and $\mu=(\mu_1\geq \cdots \geq \mu_r\geq 0)$ such that $\mu_i\leq \lambda_i$ for all $i=1,\dots,r$ and two sequence of positive integers $f=(f_1,\dots,f_r)$ and $g=(g_1,\cdots, g_r)$ such that
\begin{equation}\label{skewftab}
g_i\leq g_{i+1}, \ \  \mbox{ and }\ \ f_i\leq f_{i+1}, \ \  \mbox{whenever}\ \  \mu_i<\lambda_{i+1}. 
\end{equation}
In this case, we say that $f/g$ is a flagging of the skew shape $\lambda/\mu$.
A \emph{flagged skew set-valued tableau} of the skew shape $\lambda/\mu$ with a flagging $f/g$ is a set-valued tableau of the skew shape $\lambda/\mu$ such that each filling in the $i$-th row is a subset of $\{g_i,g_i+1,\dots,f_i\}$. Let $\FSVT(\lambda/\mu,f/g)$ denote the set of all flagged skew tableaux of shape $\lambda/\mu$ with a flagging $f/g$. If $g=(1,\dots,1)$ and $f_1=\cdots=f_r$, then the associated flagged skew set-valued tableaux are nothing but the set-valued tableaux of skew shape $\lambda/\mu$ considered by Buch in \cite{BuchLRrule}.

\begin{exm}
Consider the skew shape $\lambda/\mu$ with a flagging $f/g$ where $\lambda=(4,3,1,1)$, $\mu=(2,2,1,0)$, $f=(2,4,2,1)$, and $g=(1,2,3,1)$. The following tableaux, for example, are in $\FSVT(\lambda/\mu,f/g)$.
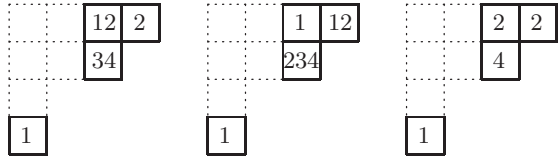
\begin{center}
\setlength{\unitlength}{0.5mm}
\begin{picture}(40,45)
\linethickness{0.7pt}
\dottedline{2}(0,40)(40,40)
\dottedline{2}(0,30)(40,30)
\dottedline{2}(0,20)(30,20)
\dottedline{2}(0,10)(10,10)
\dottedline{2}(0,0)(10,0)
\dottedline{2}(0,40)(0,0)
\dottedline{2}(10,40)(10,0)
\dottedline{2}(20,40)(20,20)
\dottedline{2}(30,40)(30,20)
\dottedline{2}(40,40)(40,30)

\put(20,40){\line(1,0){20}}
\put(20,30){\line(1,0){20}}
\put(20,20){\line(1,0){10}}

\put(20,40){\line(0,-1){20}}
\put(30,40){\line(0,-1){20}}
\put(40,40){\line(0,-1){10}}

\put(0,10){\line(1,0){10}}
\put(0,0){\line(1,0){10}}
\put(10,10){\line(0,-1){10}}
\put(0,10){\line(0,-1){10}}
\put(33,33){\footnotesize{$2$}}
\put(22,33){\footnotesize{$12$}}
\put(22,23){\footnotesize{$34$}}
\put(3,3){\footnotesize{$1$}}
\end{picture}
\ \ \ \ 
\begin{picture}(40,45)
\linethickness{0.7pt}
\dottedline{2}(0,40)(40,40)
\dottedline{2}(0,30)(40,30)
\dottedline{2}(0,20)(30,20)
\dottedline{2}(0,10)(10,10)
\dottedline{2}(0,0)(10,0)
\dottedline{2}(0,40)(0,0)
\dottedline{2}(10,40)(10,0)
\dottedline{2}(20,40)(20,20)
\dottedline{2}(30,40)(30,20)
\dottedline{2}(40,40)(40,30)

\put(20,40){\line(1,0){20}}
\put(20,30){\line(1,0){20}}
\put(20,20){\line(1,0){10}}

\put(20,40){\line(0,-1){20}}
\put(30,40){\line(0,-1){20}}
\put(40,40){\line(0,-1){10}}

\put(0,10){\line(1,0){10}}
\put(0,0){\line(1,0){10}}
\put(10,10){\line(0,-1){10}}
\put(0,10){\line(0,-1){10}}
\put(32,33){\footnotesize{$12$}}
\put(23,33){\footnotesize{$1$}}
\put(20,23){\footnotesize{$234$}}
\put(3,3){\footnotesize{$1$}}
\end{picture}
\ \ \ \ 
\begin{picture}(40,45)
\linethickness{0.7pt}
\dottedline{2}(0,40)(40,40)
\dottedline{2}(0,30)(40,30)
\dottedline{2}(0,20)(30,20)
\dottedline{2}(0,10)(10,10)
\dottedline{2}(0,0)(10,0)
\dottedline{2}(0,40)(0,0)
\dottedline{2}(10,40)(10,0)
\dottedline{2}(20,40)(20,20)
\dottedline{2}(30,40)(30,20)
\dottedline{2}(40,40)(40,30)

\put(20,40){\line(1,0){20}}
\put(20,30){\line(1,0){20}}
\put(20,20){\line(1,0){10}}

\put(20,40){\line(0,-1){20}}
\put(30,40){\line(0,-1){20}}
\put(40,40){\line(0,-1){10}}

\put(0,10){\line(1,0){10}}
\put(0,0){\line(1,0){10}}
\put(10,10){\line(0,-1){10}}
\put(0,10){\line(0,-1){10}}
\put(33,33){\footnotesize{$2$}}
\put(23,33){\footnotesize{$2$}}
\put(23,23){\footnotesize{$4$}}
\put(3,3){\footnotesize{$1$}}
\end{picture}
\end{center}
\end{exm}
The goal of this section is to show $G_{\lambda/\mu,f/g}(x)=\widetilde{G}_{\lambda/\mu,f/g}(x)$ where these functions are defined as follows. We define
\[
G_{\lambda/\mu,f/g}(x):= \sum_{T\in \FSVT(\lambda/\mu,f/g)} M(T), \ \ \ M(T):=\beta^{|T| - |\lambda/\mu|} \prod_{k\in T} x_{k},
\]
where $|T|$ is the total number of entries in $T$, $|\lambda/\mu|$ is the number of boxes in the skew shape $\lambda/\mu$, and $k\in T$ denotes an entry in $T$.
We also define
\[
\widetilde{G}_{\lambda/\mu,f/g}(x):=\det\left(\sum_{s\geq 0} \binom{i-j}{s}\beta^s G_{\lambda_i-\mu_j+j-i+s}^{[f_i/g_j]}\right)_{1\leq i,j\leq r}.
\]
where, for $p,q\in \bbN$, the function $G_m^{[p/q]}=G_m^{[p/q]}(x)$ is defined by the generating function
\[
\sum_{m\in \ZZ}G_m^{[p/q]} u^m= \frac{1}{1+\beta^{-1}u} \prod_{q\leq i\leq p} \frac{1+\beta x_i}{1-x_i u}.
\]
Note that $G_{-m}^{[p/q]} = (-\beta)^{m}$ for all integer $m\geq 0$ and, if $q>p$, then $G_m^{[p/q]}=0$ for all $m>0$. If we specialize at $\beta = 0$, $G_{\lambda/\mu,f/g}(x)$ becomes the (row) flagged skew Schur polynomial in \cite{Wachs}. Furthermore, under this specialization, $\widetilde{G}_{\lambda/\mu,f/g}(x)$ gives nothing but the corresponding Jacobi--Trudi formula also in \cite{Wachs}, since $G_m^{[p/q]}(x)$ becomes the complete symmetric function of degree $m$ in variables $x_q,x_{q+1},\dots, x_p$. 

First of all, we show the following basic formula.
\begin{lem}
If $q\leq p$, we have
\begin{eqnarray}
G_m^{[p/q]}
&=& x_q G_{m-1}^{[p/q]} + (1+\beta x_q)G_m^{[p/q+1]},\label{821i}\\
G_{m-1}^{[p/q]} +\beta G_m^{[p/q+1]}
&=&\frac{x_q}{1+\beta x_q}(G_{m-1}^{[p/q]} + \beta G_m^{[p/q]}), \label{821ii}
\end{eqnarray}
for each $m\in \ZZ$. 
\end{lem}
\begin{proof}
Equation (\ref{821i}) follows by comparing the coefficient of $u^m$ of the identity
\[
(1-x_q u)\sum_{m\in \ZZ}G_m^{[p/q]}u^m = (1+\beta x_q)\sum_{m\in \ZZ}G_m^{[p/q+1]}u^m.
\]
For (\ref{821ii}), we compute the generating function of $G_{m-1}^{[p/q]} +\beta G_m^{[p/q+1]}$:
\begin{eqnarray*}
\sum_{m\in \ZZ} (G_{m-1}^{[p/q]} + \beta G_m^{[p/q+1]} ) u^m
&=&u \frac{1}{1+\beta u^{-1}} \prod_{q\leq i\leq p} \frac{1+\beta x_i}{1-x_i u} + \beta \frac{1}{1+\beta u^{-1}} \prod_{q+1\leq i\leq p} \frac{1+\beta x_i}{1-x_i u}\\
&=&\left(\sum_{m\in \ZZ} G_m^{[p/q]} u^m\right) \left( \frac{u+ \beta }{1+\beta x_q}\right)\\
&=&\frac{1}{1+\beta x_q}\left(\sum_{m\in \ZZ} \left(G_{m-1}^{[p/q]}+\beta G_m^{[p/q]}\right) u^{m}\right).
\end{eqnarray*}
Thus Equation (\ref{821ii}) holds.
\end{proof}

We will prove that $G_{\lambda/\mu,f/g}=\widetilde{G}_{\lambda/\mu,f/g}$ using the following four propositions.
\begin{prop}\label{prop8-1}
Let $k$ be such that $\mu_k\geq \lambda_{k+1}$. 
Then we have
\begin{itemize}
\item[$(\mathrm{i})$] $\widetilde{G}_{\lambda/\mu,f/g} = \widetilde{G}_{\hat\lambda/\hat\mu,\hat{f}/\hat{g}}\cdot \widetilde{G}_{\check\lambda/\check\mu,\check{f}/\check{g}}$.
\item[$(\mathrm{ii})$] $G_{\lambda/\mu,f/g} = G_{\hat\lambda/\hat\mu,\hat{f}/\hat{g}}\cdot G_{\check\lambda/\check\mu,\check{f}/\check{g}}$.
\end{itemize}
where \ $\hat{}$ and \ $\check{}$ applied to a sequence $(t_1,\dots,t_r)$ denote the sequences $(t_1,\dots,t_k)$ and $(t_{k+1},\dots,t_r)$ respectively.
\end{prop}
\begin{proof}
The identity (ii) is trivial from the definition. We prove (i). If $j\leq k <i$, we have $\lambda_i-\mu_j\leq 0$. This implies that $(i,j)$-entry of the determinant of $\widetilde{G}_{\lambda/\mu,f/g}$ is $0$ for $j\leq k <i$. Indeed, since $G_{-m}^{[p/q]} = (-\beta)^{m}$ for $m\geq 0$ and by an identity of binomial coefficients, we have
\begin{eqnarray*}
\sum_{s\geq 0} \binom{i-j}{s}\beta^s G_{\lambda_i-\mu_j+j-i+s}^{[f_i/g_j]}
&=&\sum_{0\leq s\leq i-j} \binom{i-j}{s}\beta^s (-\beta)^{-(\lambda_i-\mu_j+j-i+s)}\\
&=&(-\beta)^{-(\lambda_i-\mu_j+j-i)}\sum_{0\leq s\leq i-j} (-1)^s\binom{i-j}{s}\\
&=&0.
\end{eqnarray*}
Thus the determinant of $\widetilde{G}_{\lambda/\mu,f/g}$ is the product of the determinants of  $\widetilde{G}_{\hat\lambda/\hat\mu,\hat{f}/\hat{g}}$ and $\widetilde{G}_{\check\lambda/\check\mu,\check{f}/\check{g}}$. 
\end{proof}

\begin{prop}\label{prop8-2}
Let $k$ be such that $\mu_k<\lambda_k$ and $g_k\leq f_k$. If $g_k<g_{k+1}$ (or $k=r$) and $\mu_{k-1}>\mu_k$ (or $k=1$), then we have
\begin{itemize}
\item[$(\mathrm{i})$] $\widetilde{G}_{\lambda/\mu,f/g}= x_{g_k}\widetilde{G}_{\lambda/\mu',f/g} + (1+\beta x_{g_k}) \widetilde{G}_{\lambda/\mu,f/g'}$,
\item[$(\mathrm{ii})$] $G_{\lambda/\mu,f/g}= x_{g_k}G_{\lambda/\mu',f/g} + (1+\beta x_{g_k})G_{\lambda/\mu,f/g'}$,
\end{itemize}
where ${}'$ applied to a sequence denotes adding $1$ to the $k$-th element of the sequence.
\end{prop}
\begin{proof}
The assumption guarantees that $f/g$ and $f/g'$ are flaggings of the skew shapes $\lambda/\mu'$ and $\lambda/\mu$ respectively. First we prove (i). The columns in the determinants of both sides are identical except for the $k$-th one. Thus the equality (i) holds if, for all $i=1,\dots,r$, we have
\begin{eqnarray}
&&\sum_{s\geq 0} \binom{i-k}{s}\beta^s G_{\lambda_i-\mu_k+k-i+s}^{[f_i/g_k]} \label{eq8-2prop}\\
&=&x_{g_k}\sum_{s\geq 0} \binom{i-k}{s}\beta^s G_{\lambda_i-\mu_k-1+k-i+s}^{[f_i/g_k]}+(1+\beta x_{g_k})\sum_{s\geq 0} \binom{i-k}{s}\beta^s G_{\lambda_i-\mu_k+k-i+s}^{[f_i/g_k+1]}.\nonumber
\end{eqnarray}
If $g_k\leq f_i$, Equation (\ref{821i}) proves the claim. Suppose that $g_k> f_i$. This implies that $f_i<f_k$ and $i\not=k$. If $i<k$, then since $\mu_k<\lambda_k \leq \lambda_i$, we have $\lambda_i - \mu_k + k-i \geq 2$. In this case, the both sides of  (\ref{eq8-2prop}) are zero since $G_m^{[f_i/g_k]}=0$ for all $m>0$. Thus it remains to show (\ref{eq8-2prop}) for the case when  $i>k$ and $\lambda_i - \mu_k + k-i \leq 1$. Since $f_i< f_k$, the condition (\ref{skewftab}) implies $\mu_k \geq \lambda_i$. 
Now we claim that (\ref{eq8-2prop}) follows by evaluating the right hand side using the identity $G_{-m}^{[f_i/g_k]}=(-\beta)^{m}$ for all $m\geq 0$. Indeed, let $a:=i-k>0$ and $b:=\mu_k-\lambda_i\geq 0$, then we have
\begin{eqnarray*}
&&x_{g_k}\sum_{s\geq 0} \binom{a}{s}\beta^s G_{-b-1-a+s}^{[f_i/g_k]}+(1+\beta x_{g_k})\sum_{s\geq 0} \binom{a}{s}\beta^s G_{-b-a+s}^{[f_i/g_k+1]}\\
&=&x_{g_k}\sum_{0\leq s\leq a} \binom{a}{s}\beta^s (-\beta)^{b+1+a-s}
+(1+\beta x_{g_k})\sum_{0\leq s\leq a} \binom{a}{s}\beta^s (-\beta)^{b+a-s}\\
&=& \sum_{0\leq s\leq a} \binom{a}{s}\beta^s (-\beta)^{b+a-s}\\
&=&\sum_{s\geq 0} \binom{a}{s}\beta^s G_{-b-a+s}^{[f_i/g_k]}.
\end{eqnarray*}
This finishes the proof of (i).

To prove (ii), we partition the set $\FSVT(\lambda/\mu, f/g)$ into three subsets $\calF_1$, $\calF_2$ and $\calF_3$: $\calF_1$ consists of those tableaux such that the filling in the leftmost box in the $k$-th row is exactly $\{g_k\}$, $\calF_2$ consists of those tableaux such that the filling in the leftmost box in the $k$-th row does not contain $x_{g_k}$, and $\calF_3$ consists of the rest. Note that by the condition (\ref{skewftab}) the set $\calF_1$ is non-empty. Since $g_{k+1}>g_k$, there is a bijection from $\calF_1$ to $\FSVT(\lambda/\mu',f/g)$ sending $T$ to $T'$ obtained from removing the leftmost box together with its filling $\{g_k\}$ from the $k$-th row. Thus $\sum_{T\in\calF_1} M(T) = x_{g_k} G_{\lambda/\mu',f/g}$. It is clear that $\calF_2 =\FSVT(\lambda/\mu,f/g')$ so that $\sum_{T\in\calF_2} M(T) = G_{\lambda/\mu,f/g'}$. Furthermore, since $g_{k+1}>g_k$, we have a bijection from $\calF_3$ to $\FSVT(\lambda/\mu,f/g')$ by removing the entry $g_k$ from the $k$-th row and leaving the rest unchagned. Then it follows that $\sum_{T\in\calF_3} M(T) = \beta x_{g_k}G_{\lambda/\mu,f/g'}$.  This proves (ii).
\end{proof}
\begin{prop}\label{prop8-3}
Let $k$ be such that $g_{k-1} = g_k$ and $\mu_{k-1} = \mu_k$. Then
\begin{itemize}
\item[$(\mathrm{i})$] $\widetilde{G}_{\lambda/\mu,f/g}= \widetilde{G}_{\lambda/\mu,f/g'}$,
\item[$(\mathrm{ii})$] $G_{\lambda/\mu,f/g}= G_{\lambda/\mu,f/g'}$,
\end{itemize}
where ${}'$ is as in Proposition \ref{prop8-2}.
\end{prop}
\begin{proof}
For (i), first we see that the determinants in the equation are identical except for the $k$-th column. We compute the difference of their $(i,k)$-entries by using (\ref{821i}) and (\ref{821ii}):
\begin{eqnarray*}
&&\sum_{s\geq 0} \binom{i-k}{s}\beta^s G_{\lambda_i-\mu_k+k-i+s}^{[f_i/g_k]}-\sum_{s\geq 0} \binom{i-k}{s}\beta^s G_{\lambda_i-\mu_k+k-i+s}^{[f_i/g_k+1]}\\
&=&\frac{x_{g_k}}{1+\beta x_{g_k}}\sum_{s\geq 0} \binom{i-k}{s}\beta^s \left(G_{{\lambda_i-\mu_k+k-i+s}-1}^{[f_i/g_k]}  + \beta G_{\lambda_i-\mu_k+k-i+s}^{[f_i/g_k]}\right)\\
&=&\frac{x_{g_k}}{1+\beta x_{g_k}}\sum_{s\geq 0} \binom{i-k}{s}\beta^s G_{{\lambda_i-\mu_k+k-i+s}-1}^{[f_i/g_k]} + \frac{x_{g_k}}{1+\beta x_{g_k}}\sum_{s\geq 0} \binom{i-k}{s-1}\beta^{s} G_{\lambda_i-\mu_k+k-i+s-1}^{[f_i/g_k]}\\
&=&\frac{x_{g_k}}{1+\beta x_{g_k}}\sum_{s\geq 0} \binom{i-k+1}{s} \beta^{s}G_{{\lambda_i-\mu_k+k-i+s}-1}^{[f_i/g_k]}
\end{eqnarray*}
Here the last equality follows from the identity of the binomial coefficients $\binom{n}{s}+\binom{n}{s-1} = \binom{n+1}{s}$ for $n,s\in \ZZ$. 
This shows that in the determiant $\widetilde{G}_{\lambda/\mu,f/g}- \widetilde{G}_{\lambda/\mu,f/g'}$, the $k$-th column coincides with $\frac{x_{g_k}}{1+\beta x_{g_k}}$ times the $(k-1)$-st column. Thus $\widetilde{G}_{\lambda/\mu,f/g}- \widetilde{G}_{\lambda/\mu,f/g'} = 0$.

For (ii), it suffices to observe that $\FSVT(\lambda/\mu,f/g) = \FSVT(\lambda/\mu,f/g')$ which follows from the column strictness.
\end{proof}
\begin{prop}\label{prop8-4}
If $f_k<g_k$ and $\mu_k<\lambda_k$ for some $k$, then
\begin{itemize}
\item[$(\mathrm{i})$] $\widetilde{G}_{\lambda/\mu,f/g}=0$,
\item[$(\mathrm{ii})$] $G_{\lambda/\mu,f/g}=0$.
\end{itemize}
\end{prop}
\begin{proof}
For (ii), it suffices to observe that $\FSVT(\lambda/\mu,f/g)=\varnothing$. We prove (i). If $r=1$, then $k=1$ and $\widetilde{G}_{\lambda/\mu,f/g}(x)=G_{\lambda_1-\mu_1}^{[f_1/g_1]} = 0$. Suppose $r> 1$. If there is $i$ such that $\mu_i \geq \lambda_{i+1}$, then the claim follows from Proposition \ref{prop8-1} and the induction hypothesis. Suppose that $\mu_i < \lambda_{i+1}$ for all $i$. In this case, (\ref{skewftab}) and the assumption imply that $\mu_i<\lambda_j$ and $f_i<g_j$ for all $i\leq k\leq j$. Then we can see that the $(i,j)$-entry of the determinant of $\widetilde{G}_{\lambda/\mu,f/g}$ is $0$ for all $i$ and $j$ such that $i\leq k\leq j$ since $G_m^{[f_i/g_j]}=0$ for all $m>0$. Thus the claim follows. 
\end{proof}
\begin{thm}
We have $\widetilde{G}_{\lambda/\mu,f/g}(x)=G_{\lambda/\mu,f/g}(x)$.
\end{thm}
\begin{proof}
With the help of Proposition \ref{prop8-1}, \ref{prop8-2}, \ref{prop8-3}, \ref{prop8-4}, the proof is exactly the same as in Theorem 3.5 \cite{Wachs}. We write the proof below for completeness. We prove this by induction on $(r, \lambda-\mu, f-g)$ ordered lexicographically where $r$ is the length of the partition $\lambda$. If $r=1$, $\widetilde{G}_{\lambda/\mu,f/g}$ and $G_{\lambda/\mu,f/g}$ are the same one row Grothendieck polynomial of degree $\lambda_1-\mu_1$ with the shifted variables, thus the claim holds. Suppose $r>1$. If $\lambda_i-\mu_i=0$ for some $i$, then we have $\lambda_{i+1}\leq \lambda_i = \mu_i$ or  $\lambda_i = \mu_i\leq \mu_{i-1}$. We apply Proposition \ref{prop8-1} for $k=i-1$ or $i$, and then the claim follows from the induction hypothesis. Suppose $\lambda_i-\mu_i>0$ for all $i=1,\dots,r$. If $f_k-g_k<0$ for some $k$, then the claim follows from Proposition \ref{prop8-4}. Suppose that $f_i-g_i>0$ for all $i=1,\dots,r$. Let $k$ be such that $g_1\geq g_2 \geq \cdots \geq g_k < g_{k+1}$ (or set $k=r$). If $\mu_{k-1}>\mu_k$ (or $k=1$), we can apply Proposition \ref{prop8-2} and the claim follows by induction. If $\mu_{k-1}=\mu_k$, then $\mu_{k-1} \leq \lambda_k$ and hence  $g_{k-1} \leq g_k$. This implies that $g_{k-1}=g_k$. Now we can apply Proposition \ref{prop8-3} and the claim follows from the induction hypothesis. 
\end{proof}

\vspace{5mm}

\noindent\textbf{Acknowledgements.} 
The author would like to thank Takeshi Ikeda for useful discussions. The author is supported by Grant-in-Aid for Young Scientists (B) 16K17584.
\bibliography{references}{}
\bibliographystyle{acm}
\end{document}